\documentclass[12pt]{article}

\usepackage[english]{babel}

\usepackage[letterpaper,top=2cm,bottom=2cm,left=3cm,right=3cm,marginparwidth=1.75cm]{geometry}

\usepackage{amsmath, amsthm, amsfonts, amssymb}
\usepackage{graphicx}
\usepackage[colorlinks=true, allcolors=blue]{hyperref}

\usepackage{verbatim, enumitem}
\usepackage{tikz}
\usetikzlibrary{arrows.meta}
\newtheorem{theorem}{Theorem}
\newtheorem{proposition}[theorem]{Proposition}
\newtheorem{lemma}[theorem]{Lemma}

\newtheorem*{theorem*}{Theorem}

\theoremstyle{definition}

\newtheorem{example}[theorem]{Example}

\newcommand{\gau}[2]{\left[ \begin{array}{c} #1 \\ #2 \end{array} \right]}
\usepackage{amsmath}
\usepackage{graphicx}
\DeclareMathOperator{\ssum}{sum}

\allowdisplaybreaks
\title{Cyclic Sieving Phenomenon for Independent sets of graphs}

\author{Jacob A White}

\begin{document}
\maketitle

\begin{abstract}
In this paper, we present examples of the cyclic sieving phenomenon coming from studying independent sets in graphs of a fixed size $k$. Given a graph $G$, and a cyclic group $C$ acting on the graph, then $C$ also acts on the collection of $k$-independent sets of G.
We exhibit cyclic sieving phenomena for a cyclic group acting on the collection of $k$-independent sets of powers of cycle graphs. As a corollary, we also find a closed formula for the number of independent sets of size $k$ in the power of a cycle graph, and in the power of a path. We also show how the graph construction of whiskering can be used to obtain new cyclic sieving phenomena from old phenomena. We also discuss recursive techniques to exhibit cyclic sieving phenomena for the independent sets of gear graphs, helm graphs, and book graphs. 
\end{abstract}

\section{Introduction}

Let $\mathfrak{C}_n$ be the cyclic group of order $n.$ Suppose that $\mathfrak{C}_n$ acts on a set $X$. For $\mathfrak{g} \in \mathfrak{C}_n$, we let $X^{\mathfrak{g}} = \{x \in X: \mathfrak{g}x=x \}.$
Given a set $X$ and a polynomial $f(q)$, we say that $(X, \mathfrak{C}_n, f(q))$ exhibits the cyclic sieving phenomenon (abbreviated CSP) if $\mathfrak{C}_n$ acts on $X$, and for any $\mathfrak{g} \in \mathfrak{C}_n$ of order $d$, and any primitive $d$th root of unity $\omega$, we have $|X^{\mathfrak{g}}| = f(\omega)$. The cyclic sieving phenomenon was first introduced in \cite{CSP1}, and many examples of CSPs appear in the literature. An excellent survey is given in \cite{CSP2}.

 In this paper we study many natural examples of cyclic sieving coming from graph theory. Throughout this paper, we only focus on simple graphs, with no loops or multiple edges. We let $V(G)$ denote the vertex set of a graph $G$.
 
Given a graph $G$, an automorphism $\mathfrak{g}: V(G) \to V(G)$ is a bijection which preserves the adjacency relation of $G$. We say that a group $\mathfrak{G}$ acts on $G$ if there is an action of $\mathfrak{G}$ on $V(G)$ that preserves the adjacency relation. We let $\mathcal{I}_k(G)$ denote the collection of independent sets of $G$ of size $k$. Given any automorphism $\mathfrak{g}$, and any $S \in \mathcal{I}_k(G)$, we observe that $\mathfrak{g}S = \{\mathfrak{g}v: v \in S \}$ is also an element of $\mathcal{I}_k(G).$ Thus if $\mathfrak{G}$ is a group that acts on $G$, then there is an induced action of $\mathfrak{G}$ on $\mathcal{I}_k(G).$ 

In particular, if a cyclic group $\mathfrak{C}$ acts on a graph $G$, then it also acts on $\mathcal{I}_k(G).$ We are interested in finding polynomials $i_k(q)$ such that $(\mathcal{I}_k(G), \mathfrak{C}, i_k(q))$ exhibits the CSP. We present several results in this regard.

Our first theorem concerns powers of cycles graphs. Given a positive integer $n$, we let $[0,n] = \{0,1,\ldots, n\}$. Given a positive integer $r$, we let $C_n^r$ be the $r$th power of the cycle graph. This graph has vertex set $[0,n-1]$ and edges $j \sim j+\ell \mbox{ mod } n$, for all $j \in [0,n-1]$ and all $\ell \in [1,r].$ We define $C_n^0$ to be the edgeless graph. The case $r=1$ is the usual cycle graph, and $r=2$ is the anti-prism graph. Several examples of $C_n^r$ are shown in Figure \ref{fig:powercycle}.

Let $\mathfrak{C}_n$ be the cyclic group of order $n$, with generator $\mathfrak{g}$. For $i \in [0,n-1] = \{0, \ldots, n-1 \}$ we define $\mathfrak{g}(i) = (i+1) \mbox{ mod } n$. This gives rise to an action of $\mathfrak{C}_n$ on $[0,n-1]$. This action preserves the adjacency relation for $C_n^r$, and hence $\mathfrak{C}_n$ acts on $\mathcal{I}_k(C_n^r)$ for any $k.$ We show that $(\mathcal{I}_k(C_n^r), \mathfrak{C}, i_{k,r,n}(q))$ exhibits the CSP for the polynomial $i_{k,r,n}(q)$ defined in Theorem \ref{thm:main}. This Theorem is proven in Section \ref{sec:main}.

Given an integer $n$, we let $[n] = 1 + q + \cdots + q^{n-1}$, and let $\gau{n}{k}$ denote the $q$-binomial coefficient. 
\begin{theorem}
Let $n$ be a positive integer and let $\mathfrak{C}_n$ be a cyclic group of order $n$ acting on $[0,n-1]$ by rotation. Let $r$ be a nonnegative integer and let $k$ be a positive integer with $(r+1)k \leq n$. 
Define $i_{k,r,n}(q) = \frac{[n]}{[n-rk]}\gau{n-rk}{k}$, which is a polynomial in $q.$
Then \[\left(\mathcal{I}_k(C_n^r), \mathfrak{C}_n, i_{k,r,n}(q)\right)\] exhibits the cyclic sieving phenomenon.
\label{thm:main}
\end{theorem}

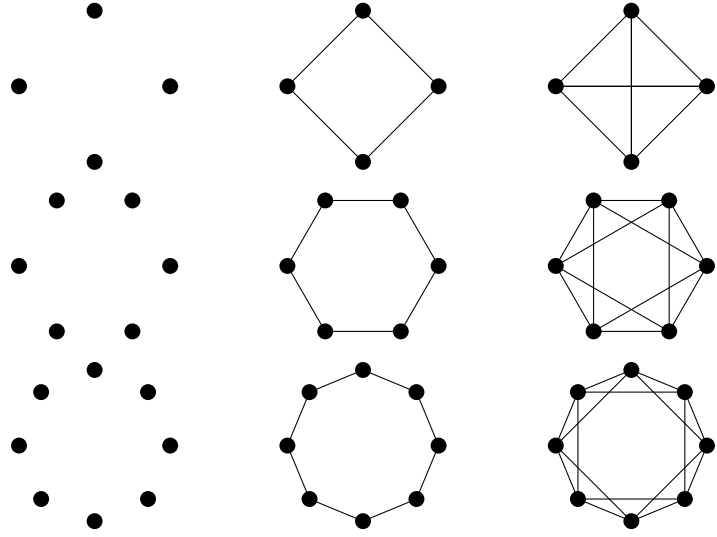
\begin{figure}
\begin{center}
\begin{tikzpicture}[scale=1]

\pgfmathtruncatemacro{\n}{4} 

\def\radius{1}               

\foreach \i in {0,...,\numexpr\n-1} {
        \coordinate (P\i) at ({\radius*cos(360*\i/\n)}, {\radius*sin(360*\i/\n)});
        \node at (P\i) {}; 
}

\foreach \i in {0,...,\numexpr\n-1} {
    \fill (P\i) circle (3pt) node[shift={(0,0.2)}] {};
}

\end{tikzpicture}
\hspace{1cm}
\begin{tikzpicture}[scale=1]

\pgfmathtruncatemacro{\n}{4} 
\pgfmathtruncatemacro{\r}{1} 

\def\radius{1}               

\foreach \i in {0,...,\numexpr\n-1} {
        \coordinate (P\i) at ({\radius*cos(360*\i/\n)}, {\radius*sin(360*\i/\n)});
        \node at (P\i) {}; 
}
\foreach \i in {0,...,\numexpr\n-1} {
    \foreach \j in {1,...,\numexpr\r} {
        \pgfmathtruncatemacro{\k}{mod(\i+\j,\n)} 
        \draw (P\i) -- (P\k);
    }
}

\foreach \i in {0,...,\numexpr\n-1} {
    \fill (P\i) circle (3pt) node[shift={(0,0.2)}] {};
}

\end{tikzpicture}
\hspace{1cm}
\begin{tikzpicture}[scale=1]

\pgfmathtruncatemacro{\n}{4} 
\pgfmathtruncatemacro{\r}{2} 

\def\radius{1}               

\foreach \i in {0,...,\numexpr\n-1} {
        \coordinate (P\i) at ({\radius*cos(360*\i/\n)}, {\radius*sin(360*\i/\n)});
        \node at (P\i) {}; 
}
\foreach \i in {0,...,\numexpr\n-1} {
    \foreach \j in {1,...,\numexpr\r} {
        \pgfmathtruncatemacro{\k}{mod(\i+\j,\n)} 
        \draw (P\i) -- (P\k);
    }
}

\foreach \i in {0,...,\numexpr\n-1} {
    \fill (P\i) circle (3pt) node[shift={(0,0.2)}] {};
}

\end{tikzpicture}

\begin{tikzpicture}[scale=1]

\pgfmathtruncatemacro{\n}{6} 

\def\radius{1}               

\foreach \i in {0,...,\numexpr\n-1} {
        \coordinate (P\i) at ({\radius*cos(360*\i/\n)}, {\radius*sin(360*\i/\n)});
        \node at (P\i) {}; 
}

\foreach \i in {0,...,\numexpr\n-1} {
    \fill (P\i) circle (3pt) node[shift={(0,0.2)}] {};
}

\end{tikzpicture}
\hspace{1cm}
\begin{tikzpicture}[scale=1]

\pgfmathtruncatemacro{\n}{6} 
\pgfmathtruncatemacro{\r}{1} 

\def\radius{1}               

\foreach \i in {0,...,\numexpr\n-1} {
        \coordinate (P\i) at ({\radius*cos(360*\i/\n)}, {\radius*sin(360*\i/\n)});
        \node at (P\i) {}; 
}
\foreach \i in {0,...,\numexpr\n-1} {
    \foreach \j in {1,...,\numexpr\r} {
        \pgfmathtruncatemacro{\k}{mod(\i+\j,\n)} 
        \draw (P\i) -- (P\k);
    }
}

\foreach \i in {0,...,\numexpr\n-1} {
    \fill (P\i) circle (3pt) node[shift={(0,0.2)}] {};
}

\end{tikzpicture}
\hspace{1cm}
\begin{tikzpicture}[scale=1]

\pgfmathtruncatemacro{\n}{6} 
\pgfmathtruncatemacro{\r}{2} 

\def\radius{1}               

\foreach \i in {0,...,\numexpr\n-1} {
        \coordinate (P\i) at ({\radius*cos(360*\i/\n)}, {\radius*sin(360*\i/\n)});
        \node at (P\i) {}; 
}
\foreach \i in {0,...,\numexpr\n-1} {
    \foreach \j in {1,...,\numexpr\r} {
        \pgfmathtruncatemacro{\k}{mod(\i+\j,\n)} 
        \draw (P\i) -- (P\k);
    }
}

\foreach \i in {0,...,\numexpr\n-1} {
    \fill (P\i) circle (3pt) node[shift={(0,0.2)}] {};
}

\end{tikzpicture}

\begin{tikzpicture}[scale=1]

\pgfmathtruncatemacro{\n}{8} 

\def\radius{1}               

\foreach \i in {0,...,\numexpr\n-1} {
        \coordinate (P\i) at ({\radius*cos(360*\i/\n)}, {\radius*sin(360*\i/\n)});
        \node at (P\i) {}; 
}

\foreach \i in {0,...,\numexpr\n-1} {
    \fill (P\i) circle (3pt) node[shift={(0,0.2)}] {};
}

\end{tikzpicture}
\hspace{1cm}
\begin{tikzpicture}[scale=1]

\pgfmathtruncatemacro{\n}{8} 
\pgfmathtruncatemacro{\r}{1} 

\def\radius{1}               

\foreach \i in {0,...,\numexpr\n-1} {
        \coordinate (P\i) at ({\radius*cos(360*\i/\n)}, {\radius*sin(360*\i/\n)});
        \node at (P\i) {}; 
}
\foreach \i in {0,...,\numexpr\n-1} {
    \foreach \j in {1,...,\numexpr\r} {
        \pgfmathtruncatemacro{\k}{mod(\i+\j,\n)} 
        \draw (P\i) -- (P\k);
    }
}

\foreach \i in {0,...,\numexpr\n-1} {
    \fill (P\i) circle (3pt) node[shift={(0,0.2)}] {};
}

\end{tikzpicture}
\hspace{1cm}
\begin{tikzpicture}[scale=1]

\pgfmathtruncatemacro{\n}{8} 
\pgfmathtruncatemacro{\r}{2} 

\def\radius{1}               

\foreach \i in {0,...,\numexpr\n-1} {
        \coordinate (P\i) at ({\radius*cos(360*\i/\n)}, {\radius*sin(360*\i/\n)});
        \node at (P\i) {}; 
}
\foreach \i in {0,...,\numexpr\n-1} {
    \foreach \j in {1,...,\numexpr\r} {
        \pgfmathtruncatemacro{\k}{mod(\i+\j,\n)} 
        \draw (P\i) -- (P\k);
    }
}

\foreach \i in {0,...,\numexpr\n-1} {
    \fill (P\i) circle (3pt) node[shift={(0,0.2)}] {};
}

\end{tikzpicture}
\caption{The graphs $C_n^r$, for $n=4,6,$ and $8$ and $r=0,1, $ and $2.$}
\label{fig:powercycle}
\end{center}
\end{figure}

As a corollary, $|\mathcal{I}_k(C_n^r)| = \frac{n}{n-rk}\binom{n-rk}{k}.$ This closed formula appears to be new for $r \geq 2$, although this sequence appears to be related to the Raney numbers \cite{raney}, which are defined by $R_{p,r}(k) = \frac{r}{pk+r}\binom{kp+r}{k}$ for positive integers $k, p,$ and $r$. If we set $r=n$, and $p=-r$, we obtain our formula.

\begin{example}
Let $n = 12$, $r=2, k=3$. We see that the orbits generated by the independent sets $\{0,3, 6\}, \{0,4,7\},$ or $\{0,5,8\}$ are free orbits of size 12, while the orbit generated by $\{0,4,8\}$ has size 4. Moreover, 
\begin{align*} \frac{[12]}{[6]} \gau{6}{3} & = (1+q^6)(1+q+2q^2+3q^3+3q^4+3q^5+3q^6+2q^7+q^8+q^9) \\ 
& \simeq 3[12]+1+q^3+q^6+q^9 (\mbox{ mod } q^{12}). \end{align*}
\end{example}

The CSP given in Theorem \ref{thm:main} has been studied before in cases $r=0$ and $r = 1$. The case $r=0$ is Theorem 1.1b in \cite{CSP1}. For $r=1$ an independent set can also be viewed as a subset of $[0,n-1]$ with no consecutive elements in the set, where we consider $n-1$ and $0$ to be consecutive. The resulting CSP is Theorem 1.2, part 2 in \cite{CSP3}. 

Our proof of the theorem is given in Section \ref{sec:main}, and will be based on expressing our polynomial as a generating function and explicit computations. We would be interested in knowing if a proof using linear algebra can be given. 

Our next theorem, which is proven in Section \ref{sec:whisker}, concerns the extension of CSPs through a graph construction. Given a graph $G$ with vertex set $V$, the \emph{whiskering} of $G$ is a new graph, denoted $w(G)$ on vertex set $V \cup V'$, where $V'$ is a set that is disjoint from $V$, and in bijection with $V.$ Given $v \in V$, we let $v'$ be the corresponding vertex in $V'.$ We require $uv \in E(w(G))$ whenever $u,v \in V$ with $uv \in E(G)$, and we also require $vv' \in E(w(G))$ for all $v \in V.$ See Figure \ref{fig:whisker} for some examples of whiskerings of graphs.

Let $\mathfrak{G}$ be a group that acts on $G.$ Then for each $v' \in V',$ and $\mathfrak{g} \in \mathfrak{G}$, we let $\mathfrak{g}(v') = (\mathfrak{g}(v))'.$ Thus we extend the action from $V$ to $V \cup V'.$ Therefore $\mathfrak{G}$ acts on $w(G)$ as well.

Our result involves instances of the CSP for whiskerings of graphs $G$ for which a cyclic group $\mathfrak{C}$ acts freely. Recall that $\mathfrak{C}$ acts freely on a set $X$ if, for every $x \in X$ and $\mathfrak{g} \in \mathfrak{C}$, we have $\mathfrak{g}x = x$ if and only if $\mathfrak{g}$ is the identity element. We say $\mathfrak{C}$ acts freely on a graph $G$ if the $\mathfrak{C}$ acts freely on $V(G)$, and this action preserves the adjacency relation of the graph.

\begin{theorem}
Let $k \leq n$. Let $G$ be a graph, and let $\mathfrak{C}$ be a cyclic group that acts freely on $V(G)$. Suppose that, for each $j \leq k$, there is a polynomial $f_j(q)$ such that $(\mathcal{I}_j(G), \mathfrak{C}, f_j(q))$ exhibits the CSP. 
Then $\mathfrak{C}$ acts on $w(G)$, and \[\left(\mathcal{I}_k(w(G)), \mathfrak{C}, \sum_{j=0}^k f_j(q) \gau{n-j}{k-j}\right)\] exhibits the CSP.
\label{thm:whisker}
\end{theorem}

Our final result gives a recursive method to finding CSPs for independent sets of graphs.
 Recall that, given a graph $G$, and $x \in G$, we have $|\mathcal{I}_k(G)| = |\mathcal{I}_k(G \setminus x)| + |\mathcal{I}_{k-1}(G \setminus N[x])|,$ where $N[x] = \{v: v \sim x \} \cup \{x\}$ is the closed neighborhood of $x$. We provide a cyclic-sieving generalization of this fact.

Let $\mathfrak{C}$ be a cyclic group acting on a graph $G$, such that there is a fixed point $x$. Then $\mathfrak{C}$ fixes $N(x)$, and thus $\mathfrak{C}$ acts on both $G \setminus x$ and $G \setminus N[x]$.
\begin{lemma}
    Let $\mathfrak{C}$ be a cyclic group acting on a graph $G$ such that $x$ is a fixed point of the action. Fix a positive integer $k$. Suppose that we have polynomials $a(q)$ and $b(q)$ such that $(\mathcal{I}_{k}(G \setminus x), \mathfrak{C}, a(q))$ and $(\mathcal{I}_{k-1}(G \setminus N[x]), \mathfrak{C}, b(q))$ both exhibit the CSP. Then $(\mathcal{I}_k(G), \mathfrak{C}, a+b)$ also exhibits the CSP.
    \label{lem:recurse}
\end{lemma}
We give a proof of the Lemma in Section \ref{sec:more}, where we also use this lemma, and the preceding theorems, to find CSPs related to gear graphs, helm graphs, and book graphs. 



\section{Cyclic Sieving for Powers of Cycle graphs}
\label{sec:main}

Our first step in proving Theorem \ref{thm:main} is to express $i_{k,r,n}(q)$ as a generating function. To this end,
we introduce a generating function $i_k(G,q)$ for graphs with vertex set $[0,n-1].$  
Given a subset $A \subseteq [0,n-1]$, and a positive integer $r$, we let
\[ \ssum_r(A) =  - (r+1)\binom{k}{2} + \sum_{a \in A} a. \]
For example, with $n = 12, k =4$, and $r=3$, we have $\ssum_{3}(\{0,5,9,11 \}) = -4(6) + 0 + 5 + 9 + 11 = 1.$

Let $G$ be a graph with vertex set $[0,n-1]$. We let $\mathcal{I}_k(G)$ denote the collection of independent sets of size $k$ of $G.$ For any $\mathcal{J} \subseteq \mathcal{I}_k(G)$, 
we define \[i(\mathcal{J},q) = \sum_{A \in \mathcal{J}} q^{\ssum_r(A)}.\] 
We let $i_k(G, q) = i(\mathcal{I}_k(G),q).$ This generating function is a Laurent polynomial. However, for the graphs $G$ we focus on, $\ssum_r(A) \geq 0$ for all $A \in \mathcal{I}_k(G)$, and thus $i_k(G,q)$ is a polynomial.

Recall that the $r$th power of the path graph, $P_n^r$, is the graph on vertex set $[0,n-1]$, with edges of the form $i \sim i+j$ for $1\leq j \leq r$ and $i \in [0,n-1].$
We see that $\mathcal{I}_2(P_5^2) = \{\{0,3\}, \{0,4\}, \{1,4\}\}$ and $i_2(P^2_5,q) = 1+q+q^2 = \gau{3}{2}.$ We have found closed formulas for $i_k(P_n^r,q)$ and $i_k(C_n^r,q).$

\begin{proposition}
Let $r \leq n$ be positive integers, and let $k$ be a positive integer such that $k \leq \frac{n}{r+1}.$
Then the following identities hold:
\begin{enumerate}
\item $i_k(P^r_n,q) = \gau{n-rk+r}{k}$,
\item $i_k(C_n^r, q) = \frac{[rk]}{[k]}i_{k-1}(P_{n-1-2r}^r,q)+q^{rk}i_k(P^r_{n-r},q)$ 
\item $i_k(C_n^r,q) = \frac{[n]}{[n-rk]}\gau{n-rk}{k}.$
\label{prop:closed}
\end{enumerate}
\end{proposition}
\begin{proof}
 We prove the first result by induction on $r$. 
 When $r=0$, then \[i_k(P_n^0,q) = \sum_{A \subset [0,n]: |A| = k} q^{\ssum_0(A)} = \gau{n}{k}. \]

For $r > 1$, let $A \in \mathcal{I}_k(P^r_n).$ Write $A = \{a_1, \ldots, a_k\}$ where $0 \leq a_1 < a_2 < \cdots < a_k \leq n-1.$ 
Define $A' = \{a_i - i+1: i \in [k] \}$. We observe that $A' \in \mathcal{I}_{k}(P_{n-k+1}^{r-1})$, and moreover $\ssum_r(A) = \ssum_{r-1}(A').$ Hence \[i_k(P_n^r,q) = i_k(P_{n-k+1}^{r-1},q) = \gau{(n-k+1)-(r-1)k+(r-1)}{k} = \gau{n-rk+r}{k}\] 
where the middle equality follows from induction. 

Now we prove our second identity.
Let $r < n$ be positive integers, and let $k$ be a positive integer such that $k(r+1) \leq n$. 
For a given $j \in [0,r-1]$, define
\[\mathcal{A}_j = \{A \in \mathcal{I}_k(C_n^r): j \in A \} \]
Define $\varphi_j: \mathcal{A}_j \to \mathcal{I}_{k-1}(P_{n-2r-1}^r)$ by $\varphi_j(A) = \{a - j-r-1: a \in A \setminus \{j\} \}$. Then $\varphi_j$ is a bijection, and we see that 
\[\ssum_r(A) = j+(j+r+1)(k-1) - (r+1)(k-1) + \ssum_r(\varphi_j(A)) = jk + \ssum_r(\varphi_j(A)).\] 
Thus \[i(\mathcal{A}_j,q) = q^{jk} i_{k-1}(P_{n-2r-1}^r,q) .\]

We let $\mathcal{B} = \{A \in \mathcal{I}_k(C_n^r): A \cap [0,r-1] = \emptyset \}$. We define $\theta:\mathcal{B} \to \mathcal{I}_k(P_{n-r}^r)$ by 
$\theta(B) = \{b - r: b \in B \}$. Then $\theta$ is a bijection, and for $B \in \mathcal{B}$ we have $\ssum_r(B) = rk + \ssum_r(\theta(B))$.
Thus, \[i(\mathcal{B},q) = q^{rk} i_k(P_{n-r}^r,q) .\]

Then we see that 
\begin{align*} i_k(C_n^r,q)  & =  \sum_{j=0}^{r-1} i(\mathcal{A}_j,q)  +  i(\mathcal{B},q) 
 =  \sum_{j=0}^{r-1} q^{jk} i_{k-1}(P_{n-2r-1}^r,q)  +  q^{rk}i_k(P_{n-r}^r,q) \\ 
 & =  i_{k-1}(P_{n-2r-1}^r,q)\sum_{j=0}^{r-1} q^{jk}  +  q^{rk}i_k(P_{n-r}^r,q)  
 =  \frac{[rk]}{[k]} i_{k-1}(P_{n-2r-1}^r,q)  +  q^{rk}i_k(P_{n-r}^r,q). \end{align*}

The third identity follows from the first two, via
\begin{align*}
i_k(C_n^r, q) & =  \frac{[rk]}{[k]} i_{k-1}(P_{n-2r-1}^r,q)  +  q^{rk}i_k(P_{n-r}^r,q) \\
 & =  \frac{[rk]}{[k]}\gau{n-2r-1-r(k-1)+r}{k-1}  +  q^{rk}\gau{n-r-rk+r}{k} \\
& =  \frac{[rk]}{[n-rk]}\gau{n-rk}{k}  +  q^{rk}\frac{[n-rk]}{[n-rk]}\gau{n-rk}{k}  \\
 & =  \frac{[rk]+q^{rk}[n-rk]}{[n-rk]}\gau{n-rk}{k}  =   \frac{[n]}{[n-rk]}\gau{n-rk}{k}.
\end{align*}
\end{proof}

We prove Theorem \ref{thm:main} by direct evaluation of $i_k(C_n^r,q)$ at a root of unity. We rely on the following lemma (Lemma 2.4 in \cite{CSP2}, and Proposition 4.2 in \cite{CSP1}):
\begin{lemma}
    Let $d \mid n$, and let $\omega$ be a primitive $d$th root of unity. Let $a$ and $b$ be positive integers with $a \equiv b \mbox{ mod } n.$ Then \[\lim_{q \to \omega} \frac{[a]}{[b]} = \begin{cases} \frac{a/d}{b/d} & d \mid a \\ 1 & d \nmid a. \end{cases}\]

    Let $k \leq n$ and suppose that $d \mid n.$ Then \[\gau{n}{k}_{q=\omega} = \binom{n/d}{k/d}\]
    \label{lem:eval}
\end{lemma}

Given a graph $G$, let $\mathfrak{G}$ be a group that acts freely on $G$. Let $\mathfrak{g} \in \mathfrak{G}$ be an automorphism of order $d$. We define $G/\mathfrak{g}$ as the graph whose vertex sets are the cycles of $\mathfrak{g}$, and where two cycles $C$ and $C'$ of $\mathfrak{g}$ are adjacent if and only if there exists $u \in C, v \in C'$ with $uv \in E(G)$.

Similarly, given $A \subseteq [0,n-1]$, with $\mathfrak{g}A = A$, we let $A/\mathfrak{g}$ be the set of cycles $C$ of $\mathfrak{g}$ such that $C \subseteq A$. The map $A \mapsto A/\mathfrak{g}$ is a bijection between $\mathcal{I}_k(G)^{\mathfrak{g}}$ and $\mathcal{I}_{k/d}(G/\mathfrak{g})$ when $d \mid k.$ Thus we obtain the following lemma:
\begin{lemma}
Let $\mathfrak{G}$ be a group that acts freely on a graph $G.$ Let 
$\mathfrak{g} \in \mathfrak{G}$ have order $d$. Let $k \leq |V(G)|.$ Then 
\[|\mathcal{I}_k(G)^{\mathfrak{g}}| = \begin{cases} |\mathcal{I}_{k/d}(G/\mathfrak{g})| & d \mid k \\ 0 & d \nmid k \end{cases}\]
\label{lem:quotient}
\end{lemma}
We state one more Lemma:
\begin{lemma}
Let $n$ be a positive integer, and let $r > 1.$
Let $\mathfrak{C}_n$ be a cyclic group of order $n$ acting on $C_n^r$ via rotation. 
Let $\mathfrak{g} \in \mathfrak{C}_n$ have order $d$. Then $C_n^r/\mathfrak{g} \simeq C_{n/d}^r.$
The isomorphism $\iota:V(C_n^r/\mathfrak{g}) \to V(C_{n/d}^r)$ is given by sending a cycle $C$ of $\mathfrak{g}$ to $\min C$, its minimum element.
\label{lem:isomorphism}
\end{lemma}
\begin{proof}
Let $\mathfrak{g}$ have order $d$, and let $C$ and $C'$ be cycles of $\mathfrak{g}$ that are adjacent in $C^r_n/\mathfrak{g}.$ Define $q$ such that $n=qd$. Let $i \in C, j \in C'$ such that $i < q$ and $i$ is adjacent to $j$.

Suppose that $i < j < i+r$. If $j < q$, then we have $i$ is adjacent to $j$ in $C_{q}^r$, and these are the minimum elements of the corresponding cycles. If $j > q$, then $i+r-q > j - q$, and $i$ is adjacent to $j-q$ in $C_q^r,$ and both are the minimum elements of the corresponding cycles.

Finally, suppose that $j+r-n > i$. We see then that $j> n-q$. Let $t = j-n+q.$ Then $t+r-q > i,$ so $i$ is adjacent to $t$ in $C_q^r$, and both $i$ and $j$ are the minimum elements of their respective cycles. Thus the map $C \mapsto \min C$ preserves edge adjacency.
\end{proof}

\begin{theorem}
Let $n$ be a positive integer, and let $k$ and $r$ be nonnegative integers such that $(k+1)r \leq n.$ Let $\mathfrak{C}_n$ be the cyclic group of order $n$, which acts on the set $[0,n-1]$ by cyclic rotation. Then $\mathfrak{C}$ acts on $\mathcal{I}_k(C^r_n)$, and 
$(\mathcal{I}_k(C^r_n), \mathfrak{C}_n, i_k(C^r_n,q))$ exhibits the CSP.
In particular, for $d \mid n$, $\mathfrak{g} \in \mathfrak{C}$ of order $d$, and $\omega$ a primitive $d$th root of unity, we have 
\[i_k(C^r_n,\omega) = |\mathcal{I}_k(C_n^r)^{\mathfrak{g}}| = \begin{cases} \frac{n/d}{n/d-rk/d}\binom{n/d-rk/d}{k/d} & d \mid k \\ 0 & d \nmid k \end{cases}.\]
\end{theorem}

\begin{proof}
The theorem is trivially true when $k = 0$, so throughout we assume $k > 0.$
Let $d \mid n$, and let $\omega$ be a primitive $d$th root of unity.
We write
\begin{equation}\frac{[n]}{[n-k]}\gau{n-rk}{k} =\frac{[n]}{[k]}\gau{n-rk-1}{k-1} = \frac{[n]}{[k]}\prod_{j=1}^{k-1} \frac{[n-rk-k+j]}{[j]}. \label{eq:factor}\end{equation}
where we are writing the factors in the numerator from smallest to largest.
We let $k = md+t$ for unique integers $m$ and $t$ with $0 \leq t < d$.  

Suppose that $t > 0$. Then the denominator on the right-hand side of Equation \eqref{eq:factor} has $m$ terms $[x]$ where $[x]_{q=\omega} = 0$, while the numerator has at least $m+1$ terms $[y]$ with $[y]_{q=\omega} = 0 $. Thus
\[\lim_{q \to \omega} \frac{[n]}{[n-rk]} \gau{n-rk}{k} = 0. \]

Suppose that $t = 0$. Then the numerator and denominator on the right-hand side of Equation \eqref{eq:factor} have exactly $m$ terms that evaluate to $0$, which cancel. If we apply Lemma \ref{lem:eval} to each ratio in the product on the right-hand side of Equation \eqref{eq:factor}, we see that, for any $j \in [k-1]$, we have 
\[ \lim_{q \to \omega} \frac{[n-rk-k+j]}{[j]} = \begin{cases} \frac{n/d-rk/d-k/d+j/d}{j/d} & d \mid j \\ 1 & d \nmid j \end{cases}. \]
 Thus we obtain
\begin{align*} 
 \lim_{q \to \omega} \frac{[n]}{[k]}\prod_{j=1}^{k-1} \frac{[n-rk-k+j]}{[j]}& = & \frac{n/d}{k/d}\prod_{j=1}^{m-1} \frac{n/d-rk/d-k/d+j}{j} & = & \frac{n/d}{k/d}\binom{n/d-rk/d-1}{k/d-1}. 
\end{align*}

Let $\mathfrak{g}$ be an element of order $d$. We know $\mathfrak{C}_n$ acts freely on $C_n^r$, and thus so does $\mathfrak{g}$. if $d \nmid k$, then $\mathcal{I}_k(C_n^r)^{\mathfrak{g}} = \emptyset$. So we assume $d | k$. Using Lemma \ref{lem:quotient} and Lemma \ref{lem:isomorphism}, we have $|\mathcal{I}_{k}(C_n^r)^{\mathfrak{g}}| = i_{k/d}(C_{n/d}^r,1)$. Applying the formula from Proposition 3, part 3 to evaluate $i_{k/d}(C_{n/d}^r,1)$, we conclude that
\[\mathcal{I}_k(C_n^r)^{\mathfrak{g}} = \begin{cases} \frac{n/d}{n/d-rk/d}\binom{n/d-rk/d}{k/d} & d \mid k \\ 0 & d \nmid k\end{cases}\]
 The result follows from applying our formula from Proposition \ref{prop:closed}, part 3 to evaluate $i_{k/d}(C_{n/d}^r,1).$  
\end{proof}

\section{Whiskering}
\label{sec:whisker}

Fix a positive integer $k$.
Suppose that a cyclic group $\mathfrak{C}$ acts on $V(G)$ freely, and suppose that for all $j \leq k$ there are polynomials $i_j(G,q)$ such that $(\mathcal{I}_j(G), \mathfrak{C}, i_j(G,q))$ exhibits the CSP. 
We know that $\mathfrak{C}$ acts on $V(w(G)).$ Recall that whiskering $G$ involves creating an isomorphic copy of $V(G)$, denoted $V'$, and adding edges of the form $v\sim v'$ for $v \in V(G)$. Figure \ref{fig:whisker} depicts whiskerings of several graphs. Given $j \leq k$, we let $\mathcal{I}_{j,k-j}(w(G)) = \{A \in \mathcal{I}_k(w(G): |A \cap V| = j \}$. We see that $\mathfrak{C}$ acts on $\mathcal{I}_{j,k-j}(w(G)).$
To prove Theorem \ref{thm:whisker}, it suffices to prove the following theorem:

\begin{theorem}
Fix positive integers $j$ and $k$ with $j \leq k.$
Suppose that a cyclic group $\mathfrak{C}$ acts on $V(G)$ freely, and suppose that there is a polynomial $i_j(G,q)$ such that $(\mathcal{I}_j(G), \mathfrak{C}, i_j(G,q))$ exhibits the CSP.
Then
$\left(\mathcal{I}_{j,k-j}(w(G)), \mathfrak{C}, i_j(G,q) \gau{n-j}{k-j}\right)$ exhibits the CSP. 
\end{theorem}

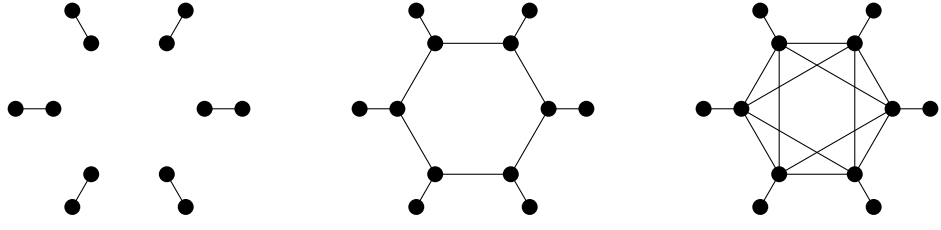
\begin{figure}
\begin{center}
\begin{tikzpicture}[scale=1]

\pgfmathtruncatemacro{\n}{6} 

\def\radius{1}               

\foreach \i in {0,...,\numexpr\n-1} {
        \coordinate (P\i) at ({\radius*cos(360*\i/\n)}, {\radius*sin(360*\i/\n)});
        \node at (P\i) {}; 
}

\foreach \i in {0,...,\numexpr\n-1} {
    \fill (P\i) circle (3pt) node[shift={(0,0.2)}] {};
}

\foreach \i in {0,...,\numexpr\n-1} {
        \coordinate (Q\i) at ({1.5*\radius*cos(360*\i/\n)}, {1.5*\radius*sin(360*\i/\n)});
        \node at (Q\i) {}; 
        \fill (Q\i) circle (3pt);
        \draw (P\i) -- (Q\i);
}

\end{tikzpicture}
\hspace{1cm}
\begin{tikzpicture}[scale=1]

\pgfmathtruncatemacro{\n}{6} 
\pgfmathtruncatemacro{\r}{1} 

\def\radius{1}               

\foreach \i in {0,...,\numexpr\n-1} {
        \coordinate (P\i) at ({\radius*cos(360*\i/\n)}, {\radius*sin(360*\i/\n)});
        \node at (P\i) {}; 
}
\foreach \i in {0,...,\numexpr\n-1} {
    \foreach \j in {1,...,\numexpr\r} {
        \pgfmathtruncatemacro{\k}{mod(\i+\j,\n)} 
        \draw (P\i) -- (P\k);
    }
}

\foreach \i in {0,...,\numexpr\n-1} {
    \fill (P\i) circle (3pt) node[shift={(0,0.2)}] {};
}

\foreach \i in {0,...,\numexpr\n-1} {
        \coordinate (Q\i) at ({1.5*\radius*cos(360*\i/\n)}, {1.5*\radius*sin(360*\i/\n)});
        \node at (Q\i) {}; 
        \fill (Q\i) circle (3pt);
        \draw (P\i) -- (Q\i);
}

\end{tikzpicture}
\hspace{1cm}
\begin{tikzpicture}[scale=1]

\pgfmathtruncatemacro{\n}{6} 
\pgfmathtruncatemacro{\r}{2} 

\def\radius{1}               

\foreach \i in {0,...,\numexpr\n-1} {
        \coordinate (P\i) at ({\radius*cos(360*\i/\n)}, {\radius*sin(360*\i/\n)});
        \node at (P\i) {}; 
}
\foreach \i in {0,...,\numexpr\n-1} {
    \foreach \j in {1,...,\numexpr\r} {
        \pgfmathtruncatemacro{\k}{mod(\i+\j,\n)} 
        \draw (P\i) -- (P\k);
    }
}

\foreach \i in {0,...,\numexpr\n-1} {
    \fill (P\i) circle (3pt) node[shift={(0,0.2)}] {};
}

\foreach \i in {0,...,\numexpr\n-1} {
        \coordinate (Q\i) at ({1.5*\radius*cos(360*\i/\n)}, {1.5*\radius*sin(360*\i/\n)});
        \node at (Q\i) {}; 
        \fill (Q\i) circle (3pt);
        \draw (P\i) -- (Q\i);
}

\end{tikzpicture}
\caption{The whiskerings of $C_6^0$, $C_6^1$, and $C_6^2$.}
\label{fig:whisker}
\end{center}
\end{figure}

\begin{proof}
Assume $\mathcal{C}$ acts freely on $G$. 
Let $\mathfrak{g} \in \mathfrak{C}$ be an element of order $d$.
 Let $\omega$ be a primitive $d$th root of unity. Suppose that $d \nmid j.$ Let $A \in \mathcal{I}_{j,k-j}(w(G)).$ Then $\mathfrak{g}^d(A \cap V) \neq A \cap V$, so $\mathcal{I}_{j,k-j}(w(G))^{\mathfrak{g}} = \emptyset.$ Similarly, $i_j(G,\omega) = 0.$

 Suppose that $d \mid j$, but $d \nmid k-j.$ For $A \in \mathcal{I}_{j,k-j}(w(G)),$ we have $\mathfrak{g}^d(A \cap V') \neq A \cap V'$, and thus $\mathcal{I}_{j,k-j}(w(G))^{\mathfrak{g}} = \emptyset.$ Moreover, $d \mid n-j, $ so $\gau{n-j}{k-j}_{q=\omega} = 0.$

 Suppose that $d | j.$ Then $i_j(G,\omega) = |\mathcal{I}_j(G)^{\mathfrak{g}}|$. Let $A \in \mathcal{I}_j(G)^{\mathfrak{g}}.$ Then $\langle \mathfrak{g} \rangle$ acts freely on $V' \setminus A'.$ Moreover, there is a $\langle \mathfrak{g} \rangle$-equivariant bijection $\sigma_A: \{B \in \mathcal{I}_{j,k-j}(G): B \cap V = A\} \to \{B \subseteq V' \setminus A': |B| = k-j \}$ given by $\sigma_A(S) = S \cap V'.$ This second set is isomorphic to $\mathcal{I}_{k-j}(C_{n-j}^0)$, and hence has size $\gau{n-j}{k-j}_{q=\omega}$. Since this latter expression does not depend on $A$, we see that $|\mathcal{I}_{k,j-k}(G)^{\mathfrak{g}}| = i_j(G,\omega)\gau{n-j}{k-j}_{q=\omega}.$
Therefore $\left(\mathcal{I}_{j,k-j}(w(G)), \mathcal{C}, i_j(G,q)\gau{n-j}{k-j} \right)$ exhibits the CSP.
\end{proof}

\begin{example}
Consider the whiskering of $C_6$, which is acted upon by the cyclic group of order six. Let $k=3$, and let $f(q) = \gau{6}{3} + \frac{[6]}{[5]}\gau{5}{1}\gau{5}{2} + \frac{[6]}{[4]}\gau{4}{2}\gau{4}{1}+\frac{[6]}{[3]}\gau{3}{3}.$ When $\omega$ is a root of unity of order two or six, then $f(\omega) = 0$, and when $\omega$ has order three, then $f(\omega) = 4$. Thus $(\mathcal{I}_3(w(C_6)), \mathfrak{C}_6, f(q))$ exhibits the CSP.
\end{example}

\section{CSPs for other graph families}
\label{sec:more}

In this section, we exhibit CSPs for independent sets of other families of graphs.

A \emph{gear graph} $G_n$ is a graph on vertex set $[0,2n]$ obtained from $C_{2n}$ by adding edges between $2n$ and $2i$ for every $i \in [0,n-1].$ The graph $G_3$ is shown in Figure \ref{fig:gear}. The cyclic group $\mathcal{C}_{2n}$ acts freely on $[0,2n-1]$ by rotation. We let $\mathfrak{C}_n$ be the subgroup of index $2$, and we extend the action to $[0,2n]$ by making $2n$ a fixed point. 

\begin{figure}
\begin{center}
\begin{tikzpicture}[scale=1]

\pgfmathtruncatemacro{\n}{6} 
\pgfmathtruncatemacro{\r}{1} 

\def\radius{1}               

\foreach \i in {0,...,\numexpr\n-1} {
        \coordinate (P\i) at ({\radius*cos(360*\i/\n)}, {\radius*sin(360*\i/\n)});
        \node at (P\i) {}; 
}
\foreach \i in {0,...,\numexpr\n-1} {
    \foreach \j in {1,...,\numexpr\r} {
        \pgfmathtruncatemacro{\k}{mod(\i+\j,\n)} 
        \draw (P\i) -- (P\k);
    }
}

\foreach \i in {0,...,\numexpr\n-1} {
    \fill (P\i) circle (3pt) node[shift={(0,0.2)}] {};
}

\coordinate (C) at (0,0);
\node at (C) {};
\fill (C) circle (3pt);

\foreach \i in {0,...,\numexpr\n-1} {
        \pgfmathtruncatemacro{\m}{mod(2*\i, \n)}
        \draw (C) -- (P\m);
}

\end{tikzpicture}
\caption{The gear graph $G_3$.}
\label{fig:gear}
\end{center}
\end{figure}
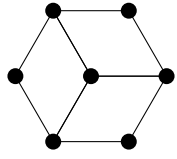

We see that Lemma \ref{lem:recurse} and Theorem \ref{thm:main} apply. 
\begin{proposition}
Let $\mathfrak{C}_n$ be a cyclic group of order $n$, acting on $G_n$ as described above. 
Then \[\left(\mathcal{I}_k(G_n), \mathfrak{C}_n, \frac{[2n]}{[2n-k]}\gau{2n-k}{k} + \gau{n}{k-1}\right)\] exhibits the CSP.
\end{proposition}
The first polynomial in the Theorem comes from the independent sets of size $k$ of $C_{2n}$, while the second polynomial comes from independent sets of size $k-1$ in $C_n^0.$ These independent sets correspond to independent sets $I$ where $2n \in I.$

\begin{example}
Let $n = 3, k = 3$. Then the cyclic group of order $3$ acts on $G_3$. We see that the orbit generated by $\{0,2,6\}$ is free of size $3$, while $\{0,2,4\}$ and $\{1,3,5\}$ are fixed points.  Let $f(q) = \frac{[6]}{[3]}\gau{3}{3} + \gau{3}{2}$. Then $f(q) = 2+q+q^2+q^3.$ We see that $f(\omega) = 2$ and $f(1) = 5$, as expected, so $\left(\mathcal{I}_2(G_3), \mathcal{C}_3, f(q)\right)$ exhibits the CSP.
\end{example}




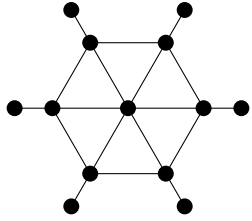
\begin{figure}
\begin{center}
\begin{tikzpicture}[scale=1]

\pgfmathtruncatemacro{\n}{6} 
\pgfmathtruncatemacro{\r}{1} 

\def\radius{1}               

\foreach \i in {0,...,\numexpr\n-1} {
        \coordinate (P\i) at ({\radius*cos(360*\i/\n)}, {\radius*sin(360*\i/\n)});
        \node at (P\i) {}; 
}
\foreach \i in {0,...,\numexpr\n-1} {
    \foreach \j in {1,...,\numexpr\r} {
        \pgfmathtruncatemacro{\k}{mod(\i+\j,\n)} 
        \draw (P\i) -- (P\k);
    }
}

\foreach \i in {0,...,\numexpr\n-1} {
    \fill (P\i) circle (3pt) node[shift={(0,0.2)}] {};
}

\foreach \i in {0,...,\numexpr\n-1} {
        \coordinate (Q\i) at ({1.5*\radius*cos(360*\i/\n)}, {1.5*\radius*sin(360*\i/\n)});
        \node at (Q\i) {}; 
        \fill (Q\i) circle (3pt);
        \draw (P\i) -- (Q\i);
}
\coordinate (C) at (0,0);
\node at (C) {};
\fill (C) circle (3pt);

\foreach \i in {0,...,\numexpr\n-1} {
        \pgfmathtruncatemacro{\m}{mod(\i, \n)}
        \draw (C) -- (P\m);
}

\end{tikzpicture}
\caption{The helm graph $H_6$.}
\label{fig:helm}
\end{center}
\end{figure}

The \emph{helm graph} $H_n$ is obtained from $w(C_n)$ by adding one new vertex $x$, which is incident to all the vertices on the cycle $C_n,$ and is not adjacent to any of the pendant vertices. We call this additional vertex the central vertex. The graph $H_6$ is given in Figure \ref{fig:helm}. Let $\mathfrak{C}_n$ be the cyclic group of order $n$, acting on $C_n$ by rotation. Then $\mathfrak{C}_n$ also acts on $w(C_n)$ and $H_n$. For the helm graph, $\mathfrak{C}_n$ fixes the central vertex $x$. We observe that $H_n \setminus N[x] = C_n^0,$ while $H_n \setminus x = w(C_n)$. Applying Theorem \ref{thm:whisker}, Theorem \ref{thm:main}, and Lemma \ref{lem:recurse}, we obtain the following proposition.

\begin{proposition}
Let $\mathfrak{C}_n$ be a cyclic group of order $n$, acting on $C_n$ via rotation. Then $\mathfrak{C}$ acts on $H_n$, and \[\left(\mathcal{I}_k(H_n), \mathfrak{C}_n, \gau{n}{k-1}+\sum_{j=0}^k \frac{[n]}{[n-j]}\gau{n-j}{j}\gau{n-j}{k-j} \right)\] exhibits the CSP.
\end{proposition}

\begin{example}
We let $n = 6$ and $k=3.$ The cyclic group of order $6$ acts on $H_6.$
Consider $f(q) = \gau{6}{2}+\gau{6}{3} + \frac{[6]}{[5]}\gau{5}{1}\gau{5}{2} + \frac{[6]}{[4]}\gau{4}{2}\gau{4}{1}+\frac{[6]}{[3]}\gau{3}{3}.$ When $\omega$ is a root of unity of order six, then $f(\omega) = 0$. We see that $f(-1) = 6,$ and when $\omega$ has order three, then $f(\omega) = 4$. Thus $(\mathcal{I}_3(H_6), \mathfrak{C}_6, f(q))$ exhibits the CSP.
\end{example}

The \emph{book graph} $B_n$ has vertices $[0,n] \times \{+1,-1\}.$ The book graph has edges $(i,j) \sim (n,j)$, and $(i,j) \sim (i,-j)$ for $i \in [0,n-1]$, and $j \in \{-1,1\}.$ We also have $(n,1) \sim (n,-1).$ The book graph $B_4$ is depicted in Figure \ref{fig:book}.

If $\mathfrak{C}$ acts freely on $[0,n-1]$, then $\mathfrak{C}$ acts on $B_n$ as follows: the vertices $(n,-1)$ and $(n,1)$ are fixed points. Otherwise, given $\mathfrak{g} \in \mathfrak{C}, i \in [0,n-1], j \in \{-1, 1\}$, we have $\mathfrak{g} \cdot (i,j) = (\mathfrak{g}\cdot i, j)$.

We see that $B_n \setminus N[(n,1)]$ is isomorphic to $C_n^0$, as is $B_n \setminus N[(n,-1)]$. Finally, $B_n \setminus \{(n,1), (n,-1)\}$ is isomorphic to $w(C^0_n).$ Using Theorem \ref{thm:main}, Theorem \ref{thm:whisker}, and Lemma \ref{lem:recurse}, we can obtain the following proposition.
\begin{proposition}
Let $\mathfrak{C}$ be a cyclic group that acts freely on $[0,n-1].$ Then $\mathfrak{C}$ acts on $B_n$, and \[\left(\mathcal{I}_k(B_n), \mathfrak{C}, 2\gau{n}{k-1}+\sum_{j=0}^k \gau{n}{j}\gau{n-j}{k-j} \right)\] exhibits the CSP.
\end{proposition}

\begin{example}
We let $n=4$ and $k=3.$
We see that the cyclic group of order 4 acts on $B_4$. We let $f(q) = 2\gau{4}{2}+\gau{4}{0}\gau{4}{3}+\gau{4}{1}\gau{3}{2}+\gau{4}{2}\gau{2}{1}+\gau{4}{3}$. Then $f(-1) = 4,$ while $f(i) = f(-i) = 0.$ Hence $(\mathcal{I}_3(B_4), \mathcal{C}_4, f(q))$ exhibits the CSP.
\end{example}

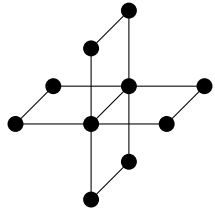
\begin{figure}
\begin{center}
\begin{tikzpicture}[scale=1]

\pgfmathtruncatemacro{\n}{4} 
\pgfmathtruncatemacro{\r}{1} 

\def\radius{1}               

\foreach \i in {0,...,\numexpr\n-1} {
        \coordinate (P\i) at ({\radius*cos(360*\i/\n)}, {\radius*sin(360*\i/\n)});
        \node at (P\i) {}; 
}

\foreach \i in {0,...,\numexpr\n-1} {
    \fill (P\i) circle (3pt) node[shift={(0,0.2)}] {};
}

\foreach \i in {0,...,\numexpr\n-1} {
        \coordinate (Q\i) at ({\radius*cos(360*\i/\n)+0.5}, {\radius*sin(360*\i/\n)+0.5});
        \node at (Q\i) {}; 
        \fill (Q\i) circle (3pt);
        \draw (P\i) -- (Q\i);
}
\coordinate (C) at (0,0);
\node at (C) {};
\fill (C) circle (3pt);
\coordinate (D) at (0.5,0.5);
\node at (D) {};
\fill (D) circle (3pt);

\foreach \i in {0,...,\numexpr\n-1} {
        \pgfmathtruncatemacro{\m}{mod(\i, \n)}
        \draw (C) -- (P\m);
        \draw (D) -- (Q\m);
}
\draw (C) -- (D);
\end{tikzpicture}
\caption{The book graph $B_4$.}
\label{fig:book}
\end{center}
\end{figure}

\section{Observations and Open Questions}
The automorphism group of $C_n^r$ is the dihedral group, and thus there are other ways in which $\mathbb{Z}/2\mathbb{Z}$ can act freely on $C_n^r,$ via reflections. In general, for those actions, $(\mathcal{I}_k(C_n^r, \mathbb{Z}/2\mathbb{Z}, i_{k,r,n}(q))$ does not exhibit the CSP.

Finally, we end with a list of questions:
\begin{enumerate}
\item Are there other commonly studied families of graph $G$ where we can find polynomials $i_k(q)$ so that $(G, \mathcal{C}, i_k(q))$ exhibits the CSP? Examples of such graphs include the prism graphs, for instance.
\item Are there other graph constructions where we can extend instances of CSP from a smaller graph to a larger graph? For example, if $\mathfrak{G}$ acts on a graph $G$, then it also acts on the bipartite double cover of $G$, and on the line graph of $G.$
\item Are there other types of graph structures, like dominating sets, for which we can find instances of the CSP?
\item The powers of cycles $C_n^r$ are examples of circulant graphs, which are graphs where the automorphism group contains a cyclic subgroup $\mathfrak{C}$ which acts transitively on $V(G)$. Given a circulant graph $G$, is there an interesting statistic $s_{G,k}$ on $\mathcal{I}_k(G)$ such that \[\left(\mathcal{I}_k(G), \mathfrak{C}, \sum\limits_{A \in \mathcal{I}_k(G)} q^{s_{G,k}(A)}\right)\]exhibits the CSP?
\end{enumerate}

\end{document}